\documentclass[10pt]{amsart}

\usepackage{amsmath,amssymb,amsthm,anysize}

\newtheorem{theorem}{Theorem}[section]
\newtheorem{lemma}[theorem]{Lemma}
\newtheorem{corollary}[theorem]{Corollary}
\newtheorem{example}[theorem]{Example}
\newtheorem{construction}[theorem]{Construction}
\newtheorem{remark}[theorem]{Remark}

\newcommand{\Aut}{\mathop{\mathrm{Aut}}}
\newcommand{\lcs}[2]{\gamma_{#2}(#1)}

\title{Point regular groups of automorphisms of generalised quadrangles}
\author{John Bamberg and Michael Giudici} 

\address{ Centre for the Mathematics of Symmetry and Computation\\
School of Mathematics and Statistics,
The University of Western Australia\\
35 Stirling Highway, Crawley, 6009 W.A.,
Australia.\\
John.Bamberg@uwa.edu.au, Michael.Giudici@uwa.edu.au}

\begin{document}

\begin{abstract}

We study the point regular groups of automorphisms of some of the known generalised quadrangles. In
particular we determine all point regular groups of automorphisms of the thick classical generalised
quadrangles. We also construct point regular groups of automorphisms of the generalised quadrangle
of order $(q-1,q+1)$ obtained by Payne derivation from the classical symplectic quadrangle
$\mathsf{W}(3,q)$. For $q=p^f$ with $f\geq 2$ we obtain at least two nonisomorphic groups when
$p\geq 5$ and at least three nonisomorphic groups when $p=2$ or $3$. Our groups include nonabelian
2-groups, groups of exponent 9 and nonspecial $p$-groups. We also enumerate all point regular groups
of automorphisms of some small generalised quadrangles.

MSC codes: 51E12, 20B15, 05E20
\end{abstract}

\maketitle

\section{Introduction}

In this paper we investigate the regular subgroups of some of the known generalised quadrangles.  We
demonstrate that the class of groups which can act as a point regular group of automorphisms of a
generalised quadrangle is much wilder than previously thought.

A finite generalised quadrangle $\mathcal{Q}$ is a geometry consisting of a finite set of points and
lines such that, if $P$ is a point and $\ell$ is a line not on $P$, then there is a unique line
through $P$ which meets $\ell$ in a point.  From this property, if there are at least three points
of $\mathcal{Q}$ or there is a point on at least three lines, then one can see that there are
constants $s$ and $t$ such that each line is incident with $s+1$ points, and each point is incident
with $t+1$ lines. Such a generalised quadrangle is said to have \emph{order} $(s,t)$, and hence its
point-line dual is a generalised quadrangle of order $(t,s)$. The generalised quadrangle is said to
be \emph{thick} if $s,t\geqslant 2$.

A permutation group $G$ on a set $\Omega$ acts \textit{regularly} on $\Omega$ if it acts
transitively on $\Omega$ and only the identity of $G$ fixes an element of $\Omega$.  Ghinelli proves
in \cite{Ghinelli} that a Frobenius group or a group with a nontrivial centre cannot act regularly
on the points of a generalised quadrangle of order $(s,s)$, where $s$ is even.  S. De Winter and
K. Thas \cite{abelian} prove that if a finite thick generalised quadrangle admits an abelian group
of automorphisms acting regularly on its points, then it is the Payne derivation of a translation
generalised quadrangle of even order. Yoshiara \cite{yoshiara} proved that there are no generalised
quadrangles of order $(s^2 , s)$ admitting an automorphism group acting regularly on points.

Our first result is a complete classification of all regular subgroups of the thick classical generalised quadrangles. 

\begin{theorem}\label{MainTheorem}
Let $\mathcal{Q}$ be a finite thick classical generalised quadrangle and let $G$ be a group of
automorphisms that acts regularly on the points of $\mathcal{Q}$.  Then one of the following holds:
\begin{enumerate}
\item $\mathcal{Q}=\mathsf{Q}^-(5,2)$ and $G$ is an extraspecial group of order 27 and exponent 3.
\item $\mathcal{Q}=\mathsf{Q}^-(5,2)$ and $G$ is an extraspecial group of order 27 and exponent 9.
\item $\mathcal{Q}=\mathsf{Q}^-(5,8)$ and $G\cong \mathrm{GU}(1,2^9).9\cong C_{513}\rtimes C_9$.
\end{enumerate}
\end{theorem}

An alternative approach to the classification in Theorem \ref{MainTheorem} was independently
undertaken in \cite{ow}.

Most of the known generalised quadrangles are \textit{elation generalised quadrangles}, and such a
generalised quadrangle $\mathcal{Q}$ of order $(s,t)$ has a group of automorphisms $G$ which fixes a
point $x$ and each line on $x$, and acts regularly on the points not collinear with $x$.  We call
$G$ an \textit{elation group} and $x$ a \textit{base point} of $\mathcal{Q}$. Necessarily, $G$ has
order $s^2t$. The only known generalised quadrangles which are not elation generalised quadrangles
are the \textit{Payne derived} quadrangles and their duals.

Payne \cite{payne1} gave a method for constructing a new generalised quadrangle from an old
one. Take a generalised quadrangle $\mathcal{Q}$ of order $(s,s)$ and suppose it has a point $x$
such that for every point $y$ not collinear with $x$, the set of points $\{x,y\}^{\perp\perp}$ has
size $s+1$, where we use the notation $S^\perp$ to denote the set of all points collinear with every
element of the set $S$.  A new generalised quadrangle $\mathcal{Q}^x$ can be constructed whose
points are the points of $\mathcal{Q}$ not collinear with $x$ and the lines of $\mathcal{Q}^x$ are:
(i) the lines of $\mathcal{Q}$ not incident with $x$, and (ii) the \textit{hyperbolic lines}
$\{x,y\}^{\perp\perp}$ where $y$ is not collinear with $x$. Thus $\mathcal{Q}^x$ is a generalised
quadrangle of order $(s-1,s+1)$.  If we take $\mathcal{Q}$ to be the elation generalised quadrangle
$\mathsf{W}(3,q)$, then any point $x$ will give rise to a Payne derived quadrangle $\mathcal{Q}^x$
of order $(q-1,q+1)$ and if $G$ is the elation group of the classical symplectic quadrangle
$\mathsf{W}(3,q)$ about the point $x$, then $G$ is elementary abelian for $q$ even and a Heisenberg
group for $q$ odd \cite{KantorSCABSGABS}.  The stabiliser $H$ of the point $x$ in the full
automorphism group of $\mathcal{Q}$ acts as a group of automorphisms of $\mathcal{Q}^x$ and contains
$G$ as a normal subgroup. In fact for $q\geq 5$, $H$ is the full automorphism group of
$\mathcal{Q}^x$ \cite{GJS}.  However, the full automorphism group of $\mathcal{Q}^x$ may contain
point regular subgroups other than $G$. In Section \ref{sec:paynederived} we exhibit several other
infinite families of regular subgroups and the results are summarised in the following theorem.

\begin{theorem}
\label{thm:secondthm}
Let $\mathcal{Q}^x$ be the generalised quadrangle of order $(q-1,q+1)$ obtained by Payne derivation
from $\mathsf{W}(3,q)$. Then there exist distinct subgroups $E$ and $P$ of $\Aut(\mathcal{Q}^x)$
that act regularly on the points of $\mathcal{Q}^x$ and for $q$ not a prime there also exists a
further regular subgroup $S$ such that $E$, $P$ and $S$ have the following properties:
\begin{enumerate}
  \item $E$ is an elation group of $\mathsf{W}(3,q)$ while $P$ and $S$ are not;
 \item $E\not\cong S$ and $P\not\cong S$;
 \item $E\cong P$ if and only if $q$ is not a power of $2$ or $3$;
 \item For $q$ even,
  \begin{enumerate}
     \item $E$ is elementary abelian while $S$ and $P$ have exponent 4 and are nonabelian
   except when $q=2$;
   \item $P'<Z(P)$ and $S'<Z(S)$ (in particular, $P$ and $S$ are not special).
  \end{enumerate}
\item For $q=3^f$, $E$ has exponent $3$ while $P$ and $S$ have exponent $9$;
\item For $q$ odd, $Z(P)=P'=Z(E)=E'$ and $Z(S)<S'$ (in particular, $P$ and $E$ are special while $S$
  is not);

\end{enumerate}
\end{theorem}

More explicit details and constructions are given in Section \ref{sec:paynederived}.  In particular,
we construct more regular subgroups than those described in Theorem \ref{thm:secondthm}; see Remark
\ref{rem:moreregularsubs}.  The generalised quadrangle of order $(2,4)$ obtained by Payne derivation from $\mathsf{W}(3,3)$ is isomorphic to $\mathsf{Q}^-(5,2)$ \cite[\S 6.1]{paynethas}. The regular groups $E$ and $P$ occurring in Theorem \ref{thm:secondthm} for this case are the two regular subgroups which appear in Theorem \ref{MainTheorem}.

The reader may notice that the existence of a regular group of
automorphisms implies that the point graph is a Cayley graph with the same automorphism group as the generalised quadrangle. Moreover, since $E$ is normal in the
full automorphism group of the Cayley graph, they are \textit{normal} Cayley graphs for $E$. However
$P$ is not normal in the full automorphism group of the Cayley graph, and so when $q$ is not a power
of $2$ or $3$, the point graph is a normal and non-normal Cayley graph for two isomorphic
groups. That this is possible answers a question posed to the authors by Yan-Quan Feng and Ted
Dobson. The only previous instance of such a phenomenon in the literature known to the authors was a
single example studied by Royle \cite{Royle}.

In Section \ref{sec:smallGQ} we list all regular subgroups of the small Payne derived generalised
quadrangles. For $q$ not a prime there are many more regular subgroups than just the groups $E,P$
and $S$ exhibited in Theorem \ref{thm:secondthm}. In Section \ref{sec:conjectures} we give an
account of how our results relate to previous results and conjectures in the literature.

\subsection*{Group theoretical terminology}

Though our group theoretic notation is standard, we briefly review it for the sake of a reader whose
interest lies more in geometry than in group theory.  We denote a cyclic group of order $n$ by
$C_n$.  If $g$ and $h$ are group elements, then we define their {\em commutator} as
$[g,h]=g^{-1}h^{-1}gh$. The \emph{centre} of a group $G$ consists of those elements $z\in G$ that
satisfy $[g,z]=1$ for all $g\in G$ and is usually denoted by $Z(G)$. For an element $g\in G$, the
\emph{centraliser} of $g$ in $G$ is the set of all elements $z\in G$ such that $[g,z]=1$ and is
denoted by $C_G(g)$. If $H,\ K$ are subgroups of a group $G$, then the \emph{commutator subgroup}
$[H,K]$ is generated by all commutators $[h,k]$ where $h\in H$ and $k\in K$. The \emph{derived
  subgroup} $G'$ of $G$ is defined as $[G,G]$ and is the smallest normal subgroup $N$ such that
$G/N$ is abelian.  The symbol $\lcs Gi$ denotes the $i$-th term of the \emph{lower central series}
of $G$; that is $\lcs G1=G$, $\lcs G2=G'$, and, for $i\ge 3$, $\lcs G{i+1}=[\lcs Gi,G]$.  The
\emph{nilpotency class} of a $p$-group is the smallest $c$ such that $\lcs G{c+1}=1$.  The
\emph{Frattini subgroup} $\Phi(G)$ of a finite group $G$ is the intersection of all the maximal
subgroups. If $G$ is a finite $p$-group, then $\Phi(G)=G'G^p$ where $G^p$ is the subgroup of $G$
generated by the $p$-th powers of all elements of $G$. In particular, $\Phi(G)$ is the smallest
normal subgroup $N$ such that $G/N$ is elementary abelian. The \emph{exponent} of a finite group $G$
is the smallest positive $n$ such that $g^n=1$ for all $g\in G$.

A $p$-group $P$ is called \emph{special} if $Z(P)=P'=\Phi(P)$. It is called \emph{extraspecial} if
it is special and these three subgroups all have order $p$. Extraspecial groups have order
$p^{1+2n}$ for some positive integer $n$ and there are two extraspecial groups of each order. When
$p$ is odd they are distinguished by their exponent: one has exponent $p$ and the other has exponent
$p^2$.  One class of special $p$-groups are the (3-dimensional) \emph{Heisenberg groups}. These are
the $p$-groups which are isomorphic to a Sylow $p$-subgroup of $\mathrm{GL}(3,q)$ for $q=p^f$, that
is, the group of lower triangular matrices with all entries on the diagonal equal to $1$. When $q$
is odd, the Heisenberg groups have exponent $p$ and for $q=p$ are extraspecial.

\section{The classical case}

The \textit{classical} generalised quadrangles are rank 2 polar spaces, whereby the points and lines
are the singular one-dimensional and two-dimensional subspaces (resp.) of a vector space equipped
with a quadratic or sesquilinear form. Below is a table listing the thick classical generalised
quadrangles together with their orders and automorphism groups:

\begin{center}
\begin{tabular}{l|l|l}
Generalised quadrangle&Order&Aut. Group\\
\hline
$\mathsf{H}(3,q^2)$ & $(q^2,q)$ & $\mathrm{P\Gamma U}(4,q)$\\
$\mathsf{H}(4,q^2)$ & $(q^2,q^3)$ & $\mathrm{P\Gamma U}(5,q)$\\
$\mathsf{W}(3,q)$ & $(q,q)$ & $\mathrm{P\Gamma Sp}(4,q)$\\
$\mathsf{Q}(4,q)$ & $(q,q)$ & $\mathrm{P\Gamma O}(5,q)$\\
$\mathsf{Q}^-(5,q)$ & $(q,q^2)$ & $\mathrm{P\Gamma O}^-(6,q)$
\end{tabular}
\end{center}

It is well-known that the first and last examples above are dual pairs and the third and fourth
examples are also dual.  In a generalised quadrangle of order $(s,t)$, a simple counting argument
shows that the number of points is $(s+1)(st+1)$ and the number of lines is $(t+1)(st+1)$.
Now we give a short proof of Theorem \ref{MainTheorem}.

\begin{proof}[Proof of Theorem \ref{MainTheorem}]
The automorphism group of a classical generalised quadrangle acts primitively on the points of the
generalised quadrangle (as it is the natural action of the classical group), and by hypothesis, it
contains a regular subgroup $G$. The classification of all regular subgroups of almost simple
primitive groups was established in the monograph \cite{regsubs} of Liebeck, Praeger and Saxl, from
which the examples in the above table are precisely the examples that arise in our context. The
groups of concern to us are dealt with in \cite[Chapters 6 and 10]{regsubs}. The regular subgroups
of $\mathrm{P\Gamma U}(n,q)$ in its action on totally isotropic $i$-spaces and some other actions of
classical groups were independently determined by Baumeister \cite{baumeister05, baumeister}. The
complete results of \cite{regsubs} require the Classification of Finite Simple Groups. However, for
an individual family of groups it only requires precise information about the subgroup structure,
and for low-dimensional classical groups this can be obtained without the Classification, for
example \cite{dimartinowag,flesner,KantLieb,mitch,mwene,wagner,zalesski}.
\end{proof}

Some other interesting consequences can be read off from the results of \cite{regsubs}, namely: (i)
only metacyclic groups can act regularly on the points of a Desarguesian projective plane (c.f.,
\cite[Chapter 5]{regsubs}); (ii) the classical generalised hexagons and octagons do not have a group
of automorphisms which act regularly on points (c.f., \cite[Chapter 12]{regsubs} and \cite[Theorem
  3]{baumeister}).

The examples below can also be found in \cite[\S10.1]{baumeister}.

\subsection*{First example: $\mathcal{Q}=\mathsf{Q}^-(5,2)$}

The regular group $G$ arising here is an extraspecial group of order 27 and exponent 3. Let $A\in
O^-(2,2)$ be of order 3. Then
$$G=\left\langle 
\begin{pmatrix} 
A&0&0\\0&A^{-1}&0\\0&0&I
\end{pmatrix},
\begin{pmatrix} 
I&0&0\\0&A&0\\0&0&A^{-1}
\end{pmatrix}, 
\begin{pmatrix}
0&I&0\\0&0&I\\I&0&0
\end{pmatrix}\right\rangle$$ 
where $G$ preserves an orthogonal decomposition of the 6-dimensional vector space into 3 anisotropic
lines. If $x$ is a nontrivial element of $G$ with 1 as an eigenvalue then 1 has multiplicity 2 and
the fixed-point space of $x$ is an anisotropic line. Thus the only nontrivial element of $G$ which
fixes a singular vector is the identity. Since the order of $G$ is equal to the number of points
(i.e., 27), $G$ is regular.

\subsection*{Second example: $\mathcal{Q}=\mathsf{Q}^-(5,2)$}

Here $G$ is an extraspecial group of order 27 and exponent 9. Let $A\in O^-(2,2)$ be of order 3. Then 
$$G=\left\langle 
\begin{pmatrix} 
0&A&0\\0&0&I\\I&0&0
\end{pmatrix},
\begin{pmatrix} 
0&I&0\\0&0&A\\I&0&0
\end{pmatrix} \right\rangle$$ 
where again $G$  preserves an orthogonal decomposition of the 6-dimensional vector space into 3 anisotropic
lines. The elements of order 9 act irreducibly on the vector space while the elements of order 3 are of the form
$$\begin{pmatrix} 
A^i&0&0\\0&A^j&0\\0&0&A^k
\end{pmatrix}$$ 
where $A^{i+j+k}=I$ with $i,j,k\in\mathbb{Z}_3$ and at least two nonzero. The fixed-point spaces of
such elements are anisotropic lines. Thus the only nontrivial element of $G$ which fixes a singular
vector is the identity. Since $|G|=27$ it follows that $G$ is regular.

\subsection*{Third example: $\mathcal{Q}=\mathsf{Q}^-(5,8)$}

Here $G\cong \mathrm{GU}(1,2^9).9\cong C_{513}\rtimes C_9$.

The number of points of $\mathcal{Q}=\mathsf{Q}^-(5,8)$ is $(8+1)(8^3+1)$, and this value is
divisible by a primitive prime divisor $19$ of $8^6-1$. The normaliser of the Sylow $19$-subgroup of
$\mathrm{P\Gamma O}^-(6,q)$ contains (and is in fact equal to) $(\mathrm{GU}(1,2^9).9):C_2$, where
the involution on top is a field automorphism. This is a typical example of a maximal subgroup of
$\mathrm{P\Gamma O}^-(6,q)$ in the ``extension field groups'' Aschbacher class.  Now the subgroup
$\mathrm{GU}(1,2^9).9$ is irreducible and has order $(8+1)(8^3+1)$. The normal subgroup
$\mathrm{GU}(1,2^9)$ is the centraliser of the Sylow $19$-subgroup and hence acts semi-regularly,
whereas the $9$ on top is an automorphism of the field extension, which ensures that
$\mathrm{GU}(1,2^9).9$ also acts semi-regularly. Therefore $\mathrm{GU}(1,2^9).9$ acts regularly on
the points of $\mathsf{Q}^-(5,8)$.

\section{Payne derived generalised quadrangles}
\label{sec:paynederived}

Let $q=p^f$ for some prime $p$ and consider the generalised quadrangle
$\mathcal{Q}=\mathsf{W}(3,q)$. Let $\textbf{x}$ be a point of $\mathcal{Q}$. As outlined in the
introduction, we can construct a new generalised quadrangle $\mathcal{Q}^x$ whose points are the
points of $\mathsf{W}(3,q)$ which are not incident with $\textbf{x}$ and whose lines of
$\mathsf{W}(3,q)$ not incident with $\textbf{x}$ together with the hyperbolic lines containing
$\textbf{x}$ but not in $\textbf{x}^{\perp}$. This generalised quadrangle is referred to as a
\emph{Payne derived quadrangle} and has order $(q-1,q+1)$.  The automorphism group of
$\mathcal{Q}^x$ contains the stabiliser of $\textbf{x}$ in $\mathrm{P}\Gamma\mathrm{Sp}(4,q)$. When
$q\geq 5$ this is the full automorphism group of $\mathcal{Q}^x$ \cite[(2.4) Corollary]{GJS}.

We will use the following setup. Let $V$ be a $4$-dimensional vector space over $\mathrm{GF}(q)$,
and consider the following alternating form on $V$:
$$\beta(x,y) := x_1y_4-y_1x_4+x_2y_3-y_2x_3.$$ The totally isotropic subspaces yield the points and
lines of the generalised quadrangle $\mathsf{W}(3,q)$, with isometry group $\mathrm{Sp}(4,q)$. Let
$\textbf{x}:=\langle (1,0,0,0)\rangle$. Then
$$\mathrm{Sp}(4,q)_{\textbf{x}}=\left\{ \begin{pmatrix}
                         \lambda & 0& 0\\
                          \textbf{u}^T   & A & 0\\
                         z & \textbf{v} & \lambda^{-1}
                     \end{pmatrix} 
              \Big\vert\,\, A\in\mathrm{GL}(2,q),\textbf{u},\textbf{v}\in \mathrm{GF}(q)^2,z\in\mathrm{GF}(q), AJA^T=J, \textbf{u}=\lambda \textbf{v} JA^T,\lambda\in\mathrm{GF}(q)\right\}$$
where $J:=\left(\begin{smallmatrix} 0&1\\-1&0 \end{smallmatrix}\right)$. 

Of particular importance is the subgroup
\begin{equation}\label{eqE}
E:=\left\{ \begin{pmatrix}
                         1 & 0& 0& 0\\
                        -c & 1& 0& 0\\
                         b & 0& 1& 0\\
                         a & b& c& 1
                     \end{pmatrix} 
              \Big\vert\,\, a,b,c\in\mathrm{GF}(q)\right\}\vartriangleleft \mathrm{Sp}(4,q)_{\textbf{x}}
\end{equation}
which has order $q^3$. Let $t_{a,b,c}$ be the element of $E$ defined by 
$$t_{a,b,c}:=\begin{pmatrix}
                         1   & 0& 0& 0\\
                        -c & 1& 0& 0\\
                         b & 0& 1& 0\\
                         a & b& c& 1
                     \end{pmatrix}.$$ 
Then a simple calculation shows that  $$  t_{a,b,c}t_{x,y,z} =t_{a+x-bz+cy,b+y,c+z} $$
for any $(a,b,c)$ and $(x,y,z)$. In particular $(t_{a,b,c})^{-1}=t_{-a,-b,-c}$ and 
\begin{equation}
 \label{eq:tcomms}
[t_{a,b,c},t_{x,y,z}]=t_{-a-x-bz+cy,-b-y,-c-z}t_{a+x-bz+cy,b+y,c+z}=t_{-2bz+2cy,0,0}
\end{equation}

We record some properties of $E$ in the following lemma.

\begin{lemma}
\label{lem:E}
Let $E$ be the group defined in (\ref{eqE}).
 \begin{enumerate}
  \item $E$ has exponent $p$.
  \item For $q$ even, $E$ is an elementary abelian $2$-group.
  \item For $q$ odd, $Z(E)=E'=\Phi(E)=\{t_{a,0,0}\mid a\in\mathrm{GF}(q)\}$.
  \item For $q$ odd, $E$ is a special group, and for $q=p$, $E$ is extraspecial of exponent $p$.
 \end{enumerate}
\end{lemma}
\begin{proof}
The first two parts follow as $(t_{a,b,c})^p=t_{0,0,0}$ for all $a,b,c\in\mathrm{GF}(q)$. For $q$
odd, we have that $Z(E)=\{t_{a,0,0}\mid a\in\mathrm{GF}(q)\}$. Since $E/Z(E)$ is elementary abelian,
$E'\leqslant Z(E)$ and equality holds since by (\ref{eq:tcomms}), each element of $Z(E)$ is a
commutator. For $p$-groups, the Frattini subgroup is the smallest normal subgroup such that the
quotient is elementary abelian, and so $\Phi(E)=E'=Z(E)$. Thus the last two parts follow.
\end{proof}

\begin{remark}\label{remark:Eiselation}
 The group $E$ acts regularly on the points of $\mathsf{W}(3,q)$ not collinear with $\textbf{x}$ and
 fixes each line through $\textbf{x}$, that is, $\mathsf{W}(3,q)$ is an elation generalised
 quadrangle with elation group $E$. Moreover, for $p$ odd, $E$ is isomorphic to the
 ($3$-dimensional) Heisenberg group.
\end{remark}

For $\alpha\in\mathrm{GF}(q)$ define
$$\theta_{\alpha}:=\begin{pmatrix}
                     1   & 0&0&0\\
                  -\alpha& 1&0&0\\
                  -\alpha^2&\alpha &1&0\\
                    0      &0&\alpha &1
                  \end{pmatrix}\in \mathrm{Sp}(4,q)_{\textbf{x}}.$$

\begin{lemma}
Let $n\geq 1$ be an integer and $\alpha\in\mathrm{GF}(q)$. Then 
$$\theta_{\alpha}^n=\begin{pmatrix}
                       1 & 0 &0&0\\
                     -n\alpha& 1& 0& 0\\
                     \frac{-n(n+1)}{2}\alpha^2 & n\alpha & 1 & 0\\
                    \frac{-n(n^2-1)}{6}\alpha^3 & \frac{n(n-1)}{2}\alpha^2 &n\alpha &1
                    \end{pmatrix}$$
\end{lemma}
\begin{proof}
The lemma is clearly true for $n=1$ so assume it is true for some $n=k-1\geq 1$. Then
$$\begin{array}{rl}
   
\theta_{\alpha}^{k}=\theta_\alpha\theta_{\alpha}^{k-1} &= 
\begin{pmatrix}
                     1   & 0&0&0\\
                  -\alpha& 1&0&0\\
                  -\alpha^2&\alpha &1&0\\
                    0      &0&\alpha &1
                  \end{pmatrix}
\begin{pmatrix}
                       1 & 0 &0&0\\
                     -(k-1)\alpha& 1& 0& 0\\
                     \frac{-(k-1)k}{2}\alpha^2 & (k-1)\alpha & 1 & 0\\
                    \frac{-(k-1)k(k-2)}{6}\alpha^3 & \frac{(k-1)(k-2)}{2}\alpha^2 &(k-1)\alpha &1
                    \end{pmatrix}  \\
       &=\begin{pmatrix}
                       1 & 0 &0&0\\
                     -k\alpha& 1& 0& 0\\
                     \frac{-k(k+1)}{2}\alpha^2 & k\alpha & 1 & 0\\
                    \frac{-k(k^2-1)}{6}\alpha^3 & \frac{k(k-1)}{2}\alpha^2 &k\alpha &1
                    \end{pmatrix}
\end{array}
$$
Thus the result follows by induction.
\end{proof}

\begin{corollary}
\label{lem:thetaorder}
 For $p>3$ and $\alpha\in\mathrm{GF}(q)\backslash\{0\}$ the element $\theta_{\alpha}$ has order $p$ while for $p=2,3$ the element $\theta_{\alpha}$ has order $p^2$. In all cases 
$$\theta_{\alpha}^{-1}=\begin{pmatrix}
                        1&0&0&0\\
                       \alpha&1&0&0\\
                       0  &-\alpha&1&0\\
                       0 &\alpha^2& -\alpha&1
                       \end{pmatrix}.$$
\end{corollary}

Let 
$$R:=\{t_{a,b,0}\mid a,b\in\mathrm{GF}(q)\}.$$ Then $R$ is an elementary abelian subgroup of $E$ of
order $q^2$. By Remark \ref{remark:Eiselation}, $R$ acts semiregularly on the set of points of
$\mathsf{W}(3,q)$ not collinear with $\textbf{x}$. Let $Z:=\{t_{a,0,0}\mid a\in\mathrm{GF}(q)\}$ and
note that $Z=Z(E)$ when $q$ is odd.

For $p=3$ and $\alpha\in\mathrm{GF}(q)\backslash\{0\}$ we have
$$\theta_{\alpha}^3=   \begin{pmatrix}
                        1&0&0&0\\
                        0&1&0&0\\
                        0  &0&1&0\\
                       -\alpha^2&0& 0&1 \end{pmatrix}\in Z(E)$$ 
while for $p=2$ we have
$$\theta_\alpha^2= \begin{pmatrix}
                        1&0&0&0\\
                        0&1&0&0\\
                        \alpha^2  &0&1&0\\
                       \alpha^3&\alpha^2& 0&1 \end{pmatrix}\in R$$ 

We collect together the following relations between the $\theta_{\alpha}$ and elements of $E$.

\begin{lemma}
\label{lem:relations}
Let $a,b,c,\alpha,\beta\in\mathrm{GF}(q)$.
\begin{enumerate}
 \item\label{1} $\theta_{\alpha}^{-1}t_{a,b,c}\theta_{\alpha}=t_{-2\alpha^2 c-2\alpha b+a,\alpha c+b, c}$
 \item\label{2} $[t_{a,b,c},\theta_{\alpha}]=t_{-\alpha(c^2+2\alpha c+2b),\alpha c,0}$
 \item\label{3} $\theta_{\alpha}\theta_{\beta}=t_{\alpha^2\beta,\alpha\beta,0}\theta_{\alpha+\beta}$
 \item\label{4} $[\theta_{\alpha},\theta_{\beta}]=t_{\alpha\beta(\alpha-\beta),0,0}$
\end{enumerate}

\end{lemma}
\begin{proof}
The first part follows as
 $$\begin{array}{rl}
  &\begin{pmatrix}
                        1&0&0&0\\
                       \alpha&1&0&0\\
                       0  &-\alpha&1&0\\
                       0 &\alpha^2& -\alpha&1
                       \end{pmatrix} \begin{pmatrix}
                1   & 0& 0& 0\\
                -c & 1& 0& 0\\
                b & 0& 1& 0\\
                a & b&c& 1
           \end{pmatrix} \begin{pmatrix}
                     1   & 0&0&0\\
                  -\alpha& 1&0&0\\
                  -\alpha^2&\alpha &1&0\\
                    0      &0&\alpha &1
                  \end{pmatrix} \\
= &\begin{pmatrix}
     1&0&0&0\\
   \alpha-c&1&0&0\\
   \alpha c +b &-\alpha&1&0\\
  -\alpha^2 c -\alpha b+a &\alpha^2+b& -\alpha+c& 1
      \end{pmatrix}  \begin{pmatrix}
                     1   & 0&0&0\\
                  -\alpha& 1&0&0\\
                  -\alpha^2&\alpha &1&0\\
                    0      &0&\alpha &1
                  \end{pmatrix} \\

= &\begin{pmatrix}
     1&0&0&0\\
     -c&1&0&0\\
     \alpha c+b&0&1&0\\
 -2\alpha^2 c-2\alpha b+a& \alpha c+b&c&1\end{pmatrix}
\end{array}.$$
The second part follows as 
$$\begin{pmatrix}
                        1&0&0&0\\
                       c&1&0&0\\
                       -b  &0&1&0\\
                       -a &-b& -c&1
                       \end{pmatrix}\begin{pmatrix}
                                         1&0&0&0\\
                                        -c&1&0&0\\
                                    \alpha c+b&0&1&0\\
                              -2\alpha^2 c-2\alpha b+a& \alpha c+b&c&1\end{pmatrix}
= \begin{pmatrix}
   1&0&0&0\\
   0&1&0&0\\
   \alpha c&0&1&0 \\
   -\alpha(c^2+2\alpha c+2b)& \alpha c&0&1
  \end{pmatrix}. $$
The third part follows from
$$\begin{array}{rl}
   \theta_{\alpha}\theta_{\beta}&=\begin{pmatrix}
                     1   & 0&0&0\\
                  -\alpha& 1&0&0\\
                  -\alpha^2&\alpha &1&0\\
                    0      &0&\alpha &1
                  \end{pmatrix} \begin{pmatrix}
                     1   & 0&0&0\\
                  -\beta& 1&0&0\\
                  -\beta^2&\beta &1&0\\
                    0      &0&\beta &1
                  \end{pmatrix} \\
       &=\begin{pmatrix}
                     1   & 0&0&0\\
                  -\alpha-\beta& 1&0&0\\
                  -\alpha^2-\alpha\beta-\beta^2&\alpha+\beta &1&0\\
                  -\alpha\beta^2  &\alpha\beta&\alpha+\beta &1
          \end{pmatrix}\\
       &= \begin{pmatrix}
            1   & 0&0&0\\
            0& 1&0&0\\
           \alpha\beta&0&1&0\\
           \alpha^2\beta&\alpha\beta&0&1
         \end{pmatrix}
         \begin{pmatrix}
             1   & 0&0&0\\
                  -(\alpha+\beta)& 1&0&0\\
                  -(\alpha+\beta)^2&\alpha+\beta &1&0\\
                   0 &\ 0&\alpha+\beta &1
         \end{pmatrix} \end{array}.$$

Finally,
$$\begin{array}{rl}
  
\theta_{\alpha}^{-1}\theta_{\beta}^{-1}\theta_{\alpha}\theta_{\beta} &
\begin{pmatrix}
                        1&0&0&0\\
                       \alpha&1&0&0\\
                       0  &-\alpha&1&0\\
                       0 &\alpha^2& -\alpha&1
                       \end{pmatrix}
\begin{pmatrix}
                        1&0&0&0\\
                       \beta&1&0&0\\
                       0  &-\beta&1&0\\
                       0 &\beta^2& -\beta&1
                       \end{pmatrix}
\begin{pmatrix}
                     1   & 0&0&0\\
                  -\alpha-\beta& 1&0&0\\
                  -\alpha^2-\alpha\beta-\beta^2&\alpha+\beta &1&0\\
                  -\alpha\beta^2  &\alpha\beta&\alpha+\beta &1
          \end{pmatrix}\\

&= \begin{pmatrix}
            1   & 0&0&0\\
        \alpha+\beta&1&0&0\\
       -\alpha\beta&-\alpha-\beta&1&0\\
       \alpha^2\beta&\alpha^2+\alpha\beta+\beta^2&-\alpha-\beta&1
   \end{pmatrix}
\begin{pmatrix}
                     1   & 0&0&0\\
                  -\alpha-\beta& 1&0&0\\
                  -\alpha^2-\alpha\beta-\beta^2&\alpha+\beta &1&0\\
                  -\alpha\beta^2  &\alpha\beta&\alpha+\beta &1
          \end{pmatrix} \\

 &= \begin{pmatrix}
                     1   & 0&0&0\\
                   0 & 1& 0& 0\\
                   0 & 0& 1& 0\\
                \alpha\beta(\alpha-\beta) & 0&0& 1
           \end{pmatrix}
\end{array}.
$$
\end{proof}

We will need the following lemma.
\begin{lemma}
 \label{lem:GLsylowp}
Let $S$ be the Sylow $p$-subgroup of $\mathrm{GL}(4,q)$ for $q=p^f$. If $p>3$ then $S$ has exponent
$p$, while if $p=2,3$ then $S$ has exponent $p^2$.
\end{lemma}

\begin{proof}
Let $g\in S$. Then the Jordan blocks of $g$ have sizes $1,2,3$ or $4$. If $\ell\leq p$, a Jordan
block of size $\ell$ has order $p$ as an element of $\mathrm{GL}(\ell,q)$. Hence if $p\geq 5$ then $g$ has exponent $p$. When $p=3$, Jordan
blocks of size 4 have order 9, while for $p=2$, Jordan blocks of size 3 and 4 have order 4. Hence
for $p=2,3$, $S$ has exponent $p^2$.
\end{proof}

We will also need the following.

\begin{lemma}
\label{lem:field}
Let $q=2^f$ with $f\geq 2$. Then $\{\alpha\beta(\alpha+\beta)\mid
\alpha,\beta\in\mathrm{GF}(q)\}=\mathrm{GF}(q)$ for $f\geq 3$ and $\{\alpha\beta(\alpha+\beta)\mid
\alpha,\beta\in\mathrm{GF}(4)\}=\{0,1\}$.
\end{lemma}

\begin{proof}
Let $S=\{\alpha\beta(\alpha+\beta)\mid \alpha,\beta\in\mathrm{GF}(q)\}$. We can easily check that
$S=\{0,1\}$ when $q=4$ so we may assume that $f\geq 3$. Clearly $S\ne\{0\}$ and so there exists
$x,y\in\mathrm{GF}(q)$ such that $xy(x+y)=a\neq 0$. Then for all $\omega\in \mathrm{GF}(q)$,
$(\omega x)(\omega y)(\omega x+\omega y)=\omega^3 a$. Thus if $a\in S$ then $S$ contains $\{\omega^3
a\mid\omega\in\mathrm{GF}(q)\}$. Let $T$ be the set of nonzero cubes in $\mathrm{GF}(q)$. If
$T=\mathrm{GF}(q)\backslash\{0\}$ then $S=\mathrm{GF}(q)$ so we may assume that $f$ is even. Then
the set of nonzero elements of $\mathrm{GF}(q)$ can be partitioned into the three sets $T$, $\xi T$
and $\xi^2 T$, where $\xi$ is a primitive element of $\mathrm{GF}(q)$. It remains to show that $S$
contains at least one element from each of these three sets, and then the result will follow. Now
$S$ contains the subset $X=\{\alpha^2+\alpha\mid\alpha\in\mathrm{GF}(q)\}$. The map
$\alpha\mapsto\alpha^2+\alpha$ is $\mathrm{GF}(2)$-linear with kernel $\mathrm{GF}(2)$. Hence
$|X|=2^{f-1}$. For $f\geq 3$, we have $|X|>|T|$ and so $X$ meets at least two of the sets $T$, $\xi
T$ and $\xi^2 T$. Thus $|S|\geq 2(2^f-1)/3>|X|$ and so there exists $\mu\in\mathrm{GF}(q)$ such that
the image $Y$ of the $\mathrm{GF}(2)$-linear map $\alpha\mapsto \alpha^2\mu+\alpha\mu^2$ is not $X$.
Then as $Y$ is another $\mathrm{GF}(2)$-subspace of $\mathrm{GF}(q)$ not equal to $X$ it follows
that $|X\cap Y|=2^{f-2}$ and so $|X\cup Y|=2^f-2^{f-2}>2(2^f-1)/3$. Hence $S$ meets each of $T$,
$\xi T$ and $\xi^2 T$ and so $S=\mathrm{GF}(q)$.
\end{proof}

\begin{construction}
\label{con:P}
Let $\{\alpha_1,\alpha_2,\ldots,\alpha_f\}$ be a basis for $\mathrm{GF}(q)$ over
$\mathrm{GF}(p)$. Let
$$P:=\langle R,\theta_{\alpha_1},\ldots,\theta_{\alpha_f}\rangle.$$ Note that $P$ is independent of
the choice of basis $\{\alpha_1,\alpha_2,\ldots,\alpha_f\}$ of $\mathrm{GF}(q)$ over
$\mathrm{GF}(p)$ since, by Lemma \ref{lem:relations}(\ref{3}), $P$ contains $\theta_{\alpha}$ for
all $\alpha\in\mathrm{GF}(q)$.
\end{construction}

\begin{lemma}
\label{lem:P}
The group $P$ has order $q^3$ and has the following properties:
\begin{enumerate}
 \item $P/R$ is elementary abelian of order $q$;
 \item For $q>2$, $P$ is nonabelian;
 \item For $q=2$, $P\cong C_4\times C_2$;
 \item For $q$ odd, $Z(P)=P'=\Phi(P)=Z(E)$;
 \item For $q>2$ even $Z(P)=R$. Moreover, $P'=Z$ for $q\geq 8$ and $P'=\{t_{0,0,0},t_{1,0,0}\}$ for
   $q=4$.
 \item For $p>3$, $P$ has exponent $p$; 
 \item For $p=2,3$, $P$ has exponent $p^2$.
\end{enumerate}
\end{lemma}

\begin{proof}
By Lemma \ref{lem:relations}(\ref{1}) each $\theta_{\alpha_i}$ normalises $R$, and by Lemma
\ref{lem:relations}(\ref{3}), $\theta_\alpha\theta_\beta\in R\theta_{\alpha+\beta}$. Thus $|P|=q^3$
and $P/R$ is isomorphic to the additive group of $\mathrm{GF}(q)$ (and so (1) holds). Hence
$P'\leqslant R$.  It follows from Lemma \ref{lem:relations}(\ref{1}) that for $q>2$, $P$ is
nonabelian (and so (2) holds). Moreover, for $q$ odd, $C_R(\theta_{\alpha})=Z(E)$ and it follows
that $Z(P)=Z(E)$. For $q$ even we have $C_R(\theta_{\alpha})=R$ and if $q>2$ it follows that
$Z(P)=R$. For $q=2$, $P$ is an abelian group of exponent 4 and so is isomorphic to $C_4\times
C_2$. Thus (3) holds.

When $q$ is odd, Lemma \ref{lem:relations}(\ref{1}) and (3) implies that $P/Z(E)$ is an elementary
abelian group of order $q^2$ and so $P'\leqslant Z(E)$. Moreover, by Lemma \ref{lem:relations}(2)
each element of $Z(E)$ is a commutator of elements of $P$. Thus $P'=\Phi(P)=Z(E)$ and so (4) holds.

For $q$ even, Lemma \ref{lem:relations}(4) implies that for $q>2$ we have $X= \langle
t_{\alpha\beta(\alpha+\beta)0,0}\mid\alpha,\beta\in\mathrm{GF}(q)\rangle\leqslant P'$. Lemma
\ref{lem:field} implies that $X=Z$ when $q\geq 8$ while $X=\{t_{0,0,0},t_{1,0,0}\}$ for $q=4$. Since
$P/X$ is abelian (Lemma \ref{lem:relations}(\ref{2}) and (\ref{4})), it follows that $P'=X$ and
hence (5) holds.

Since $P$ is contained in a Sylow $p$-subgroup of $\mathrm{GL}(4,q)$, Lemma \ref{lem:GLsylowp}
implies that the exponent of $P$ is $p$ for $p\geq 5$ (and (6) holds) and at most $p^2$ for
$p=2,3$. For $p=2,3$, each $\theta_{\alpha}\in P$ has order $p^2$ and so the exponent of $P$ is
indeed $p^2$. Therefore, (7) holds.
\end{proof}

\begin{corollary}
\label{cor:PnotE}
 For $p=2,3$, $P\not\cong E$.
\end{corollary}
\begin{proof}
 This follows by comparing the exponents of $E$ and $P$.
\end{proof}

\begin{lemma}
 For $q=p^f$ with $p>3$, $E\cong P$.
\end{lemma}
\begin{proof}
Define the map $$\begin{array}{rclll}
           \phi:&P&\rightarrow &E&\\
                &t_{a,b,0}&\mapsto& t_{a,b,0}& a,b\in\mathrm{GF}(q)\\
                &\theta_{\alpha_i}&\mapsto &t_{0,-\alpha_i^2/2,\alpha_i} &i=1,\ldots,f
         \end{array}$$
which maps a set of generators of $P$ to a set of generators for $E$. The generators for each group
have order $p$. Since
$[\theta_{\alpha},\theta_\beta]=t_{\alpha\beta(\alpha-\beta),0,0}=[t_{0,-\alpha^2/2,\alpha},t_{0,-\beta^2/2,\beta}]$
and $[t_{a,b,0},\theta_\alpha]=t_{-2\alpha b,0,0}=[t_{a,b,0},t_{0,-\alpha^2/2,\alpha}]$ this map
extends to an isomorphism.
\end{proof}

\begin{lemma}
\label{lem:Pregular}
The group $P$ acts regularly on the set of points of $\mathsf{W}(3,q)$ not collinear with
$\textbf{x}$. Moreover, $P$ fixes the line $\langle (1,0,0,0),(0,1,0,0)\rangle$ but transitively
permutes the remaining $q$ lines through $\textbf{x}$.
\end{lemma}
\begin{proof}
Consider the image of $\textbf{y}=\langle (0,0,0,1)\rangle$ under $g\in P$. Then
$g=t_{a,b,0}\theta_{\alpha_1}^{n_1}\theta_{\alpha_2}^{n_2}\ldots\theta_{\alpha_f}^{n_f}$ for some
$t_{a,b,0}\in R$ and integers $n_1,\ldots,n_f$. The third coordinate of $(0,0,0,1)^g$ is equal to
$n_1\alpha_1+\cdots+ n_f\alpha_f$.  Thus if $g\in P_{\textbf{y}}$ then $n_1\alpha_1+\cdots
+n_f\alpha_f=0$. Since $\{\alpha_1,\ldots,\alpha_f\}$ is a linearly independent set over
$\mathrm{GF}(p)$ it follows that $p$ divides $n_i$ for all $i$. Hence for each $i$,
$\theta_{\alpha_i}^{n_i}\in R$ and so $g\in R_{\textbf{y}}$. However, $R$ acts semiregularly on the
set of points not collinear with $\textbf{x}$. Thus $g=1$ and since $|P|=q^3$, we have that $P$ acts
regularly on the set of points of $\mathsf{W}(3,q)$ not collinear with $\textbf{x}$.

Since each $\theta_{\alpha_i}$ induces $\begin{pmatrix}
                                        1&0\\\alpha_i&1
                                       \end{pmatrix}$ 
on $\textbf{x}^{\perp}/\textbf{x}$ it follows that $P$ fixes one line through $\textbf{x}$ and
transitively permutes the remaining $q$.
\end{proof}

For $q=3$ the groups $E$ and $P$ are the two regular subgroups of $Q^-(5,2)$ given in Theorem \ref{MainTheorem}.

\begin{lemma}
\label{lem:unique}
Let $\mathcal{Q}^x$ be the generalised quadrangle obtained by Payne derivation from
$\mathsf{W}(3,p)$ for $p\geq 5$ a prime and suppose that $G$ acts regularly on the set of points of
$\mathcal{Q}^x$. Then $G\cong E\cong P$.
\end{lemma}

\begin{proof}
 Since $|G|=p^3$ and is contained in a Sylow $p$-subgroup of $\mathrm{GL}(4,p)$, by Lemma
 \ref{lem:GLsylowp}, $G$ has exponent $p$. Thus by inspecting the five groups of order $p^3$ (namely
 $C_{p^3}$, $C_{p^2}\times C_p$, $C_p^3$ and the two extraspecial groups) we deduce that either
 $G\cong E$ or $G$ is elementary abelian.  By \cite[Main Theorem 2.6]{abelian} the latter is not
 possible.
\end{proof}

\begin{construction}
 \label{con:SUW}
 For $q=p^f$ with $f\geq 2$, let $U\oplus W$ be a decomposition of $\mathrm{GF}(q)$ into
 $\mathrm{GF}(p)$-subspaces and let $\{\alpha_1,\ldots,\alpha_k\}$ be a basis for $U$. Define
$$S_{U,W}=\langle R,\theta_{\alpha_1},\ldots,\theta_{\alpha_k},t_{0,0,w}\mid w\in W\rangle$$ Note
 that $S_{U,W}$ is independent of the choice of basis $\{\alpha_1,\alpha_2,\ldots,\alpha_k\}$ of $U$
 over $\mathrm{GF}(p)$ since, by Lemma \ref{lem:relations}(\ref{3}), $S_{U,W}$ contains
 $\theta_{\alpha}$ for all $\alpha\in U$.
\end{construction}

\begin{lemma}
\label{lem:SUW}
The group $S_{U,W}$ given by Construction \ref{con:SUW} has order $q^3$ and has the following properties:
\begin{enumerate}
 \item $S_{U,W}$ is nonabelian;
 \item for $q$ odd, $Z(S_{U,W})=Z(E)$ while 
$$(S_{U,W})'=\langle Z(E), t_{0,\alpha_1w_1+\cdots+\alpha_kw_k,0}\mid w_i\in W\rangle$$ which has
   order $qp^\ell$ where $\ell=\dim_{\mathrm{GF}(p)}(\alpha_1W+\cdots+\alpha_kW)$.
 \item for $q$ even, $Z(S_{U,W})=R$ and 
$$  (S_{U,W})' =\left\{ \begin{array}{ll}   \langle Z,t_{\alpha_1 w_1^2+\cdots+\alpha_kw_k^2,\alpha_1 w_1+\cdots+\alpha_kw_k,0}\mid w_i\in W\rangle & \text{ for } k\geq 3\\
            \langle t_{1,0,0},t_{\alpha_1 w_1^2+\alpha_2w_2^2,\alpha_1 w_1+\alpha_2w_2,0}\mid w_i\in W\rangle & \text{ for } k= 2\\
             \langle t_{\alpha_1 w_1^2,\alpha_1 w_1}\mid w_i\in W\rangle & \text{ for } k= 1  \end{array} \right.$$
\item for $p>3$, $S_{U,W}$ has exponent $p$;
 \item for $p=2,3$, $S_{U,W}$ has exponent $p^2$.
\end{enumerate}
\end{lemma}
\begin{proof}
Now $(\theta_{\alpha_i})^p\in R$ and by Lemma \ref{lem:relations}(\ref{1}), each $\theta_{\alpha_i}$
normalises $\left\langle R,t_{0,0,w} \mid w\in W\right\rangle$. Hence $|S_{U,W}|=q^3$. Since
$\theta_{\alpha_i}$ does not centralise elements $t_{0,0,w}$ for $w\in W\backslash\{0\}$, it follows
that $S_{U,W}$ is nonabelian. Hence (1) holds.

For $q$ odd we have that $Z(S_{U,W}\cap E)=Z(E)$. Since $C_E(\theta_{\alpha_i})=Z(E)$ it follows
that $Z(S_{U,W})=Z(E)$. For $q$ even we have that $S_{U,W}\cap E$ is elementary abelian. Moreover,
by Lemma \ref{lem:relations}(\ref{1}), $\theta_{\alpha_i}$ centralises $R$ and so $R\leqslant
Z(S_{U,W})$. Since $\theta_{\alpha_i}$ does not centralise any element of $S_{U,W}\cap E$ outside
$R$ it follows that $Z(S_{U,W})=R$.

By Lemma \ref{lem:relations}(\ref{1}),(\ref{3}) and (\ref{4}), $S_{U,W}/R$ is elementary abelian of
order $q$ and so $(S_{U,W})'\leqslant R$. For $q$ odd all elements of $Z(E)$ can be written as
commutators of $\theta_{\alpha_1}$ and elements of $R$ (Lemma \ref{lem:relations}(\ref{2})). Hence
$Z(E)\leqslant (S_{U,W})'$. Moreover, $[\theta_{\alpha_i},t_{0,0,w}]=t_{-\alpha_i(w^2+2\alpha_i
  w),\alpha_i w,0}$ for all $w\in W$. Thus
$$X=\langle Z(E), t_{0,\alpha_1w_1+\cdots+\alpha_kw_k,0}\mid w_i\in W\rangle\leqslant (S_{U,W})'.$$
By Lemma \ref{lem:relations}(\ref{1}) and (\ref{3}), $S_{U,W}/X$ is abelian and so $(S_{U,W})'=X$,
which has order $qp^\ell$. Thus (2) holds.

For $q$ even, by Lemma \ref{lem:relations}(\ref{2}) and (\ref{4}),
$[\theta_{\alpha_i},t_{0,0,w}]=t_{\alpha_i w^2,\alpha_i w,0}$ and
$[\theta_{\alpha_i},\theta_{\alpha_j}]=t_{\alpha_i\alpha_j(\alpha_i-\alpha_j),0,0}$. Thus
$$X=\langle t_{\alpha\beta(\alpha+\beta),0,0},t_{\alpha_1 w_1^2+\cdots+\alpha_kw_k^2,\alpha_1
  w_1+\cdots+\alpha_kw_k,0}\mid w_i\in W,\alpha,\beta\in\mathrm{GF}(q)\rangle \leqslant (S_{U,W})'$$
Moreover, by Lemma \ref{lem:relations}(\ref{1}),
$\theta_{\alpha_i}^{-1}t_{a,b,c}\theta_{\alpha_i}=t_{a,\alpha_i c+b, c}=t_{\alpha_i c^2,\alpha_i
  c,0}t_{a,b,c}$. Since $E$ is abelian and $t_{a,b,c}\in S_{U,W}$ if and only if $c\in W$, it
follows that $S_{U,W}/X$ is elementary abelian. Hence $(S_{U, W})'=X$.  The expression for
$(S_{U,W})'$ given in (3) then follows from Lemma \ref{lem:field}.

Since $S_{U,W}$ is contained in a Sylow $p$-subgroup of $\mathrm{GL}(4,q)$, Lemma \ref{lem:GLsylowp}
implies that the exponent of $S_{U,W}$ is $p$ for $p\geq 5$ and at most $p^2$ for $p=2,3$. For
$p=2,3$, each $\theta_{\alpha_i}\in S_{U,W}$ has order $p^2$ and so the exponent of $S_{U,W}$ is
indeed $p^2$. Hence (4) and (5) hold.
\end{proof}

\begin{corollary}
\label{cor:SUPnotPE}
 Let $U$ be a 1-dimensional subspace of $\mathrm{GF}(q)$ over $\mathrm{GF}(p)$. Then $P\not\cong
 S_{U,W}\not\cong E$.
\end{corollary}

\begin{proof}
This follows for $q>4$ by comparing the orders of the derived subgroups. Note that
$\ell=\dim(W)=f-1$. A Magma \cite{magma} calculation verifies the result for $q=4$.
\end{proof}

\begin{remark}\label{rem:moreregularsubs}
Note that $|S_{U,W}\cap E|=q^2p^{f-k}$. Since $E\vartriangleleft
\Gamma\mathrm{Sp}(4,q)_{\textbf{x}}$ it follows that if $\dim(U_1)\neq\dim(U_2)$ then $S_{U_1,W_1}$
is not conjugate to $S_{U_2,W_2}$ in $ \Gamma\mathrm{Sp}(4,q)_{\textbf{x}}$. However, if
$\dim(U_1)=\dim(U_2)$ it is possible for $S_{U_1,W_1}$ to still not be conjugate to
$S_{U_2,W_2}$. For example, when $q=8$, Magma \cite{magma} calculations show that there are two
conjugacy classes of subgroups $S_{U,W}$ with $U$ a 2-dimensional subspace.

As for isomorphism classes, sometimes it can be read off from the order of derived subgroups that
two such groups are nonisomorphic. For example when $q=8$ comparing orders of derived subgroups
yields $S_{U,W}\not\cong S_{W,U}$ when $U$ is a 1-space. Moreover, if $U$ is a 2-space then
$S_{U,W}\not\cong P$ even though they have derived subgroups of the same order. It is even possible
for $S_{U_1,W_1}\not\cong S_{U_2,W_2}$ when $\dim(U_1)=\dim(U_2)$. Indeed for $q=16$, Magma
\cite{magma} calculations show that there are two isomorphism classes of subgroups $S_{U,W}$ with
$U$ a 3-dimensional subspace: one has Frattini subgroup of order $2^7$ and one has Frattini subgroup
of order $2^8$.
\end{remark}

\begin{lemma}
\label{lem:SUWregular}
The group $S_{U,W}$ acts regularly on the set of points of $\mathsf{W}(3,q)$ not collinear with
$\textbf{x}$. Moreover, $S_{U,W}$ fixes the line $\langle (1,0,0,0),(0,1,0,0)\rangle$ but permutes
the remaining $q$ lines through $\textbf{x}$ in orbits of length $p^k$.
\end{lemma}
\begin{proof}
Consider the image of $\textbf{y}=\langle (0,0,0,1)\rangle$ under $g\in S_{U,W}$. Then
$g=t_{a,b,c}\theta_{\alpha_1}^{n_1}\theta_{\alpha_2}^{n_2}\ldots\theta_{\alpha_k}^{n_k}$ for some
$t_{a,b,c}\in E\cap S_{U,W}$ and some integers $n_1$. The third coordinate of $(0,0,0,1)^g$ is equal
to $n_1\alpha_1+\cdots +n_k\alpha_k+c$ where $c\in W$.  Thus if $g\in (S_{U,W})_{\textbf{y}}$ then
$n_1\alpha_1+\cdots +n_k\alpha_k+c=0$. Since $\{\alpha_1,\ldots,\alpha_k,c\}$ is a linearly
independent set over $\mathrm{GF}(p)$ it follows that $p$ divides each $n_i$. Hence each
$\theta_{\alpha_i}^{n}\in R$ and so $g\in R_{\textbf{y}}$. However, $R$ acts semiregularly on the
set of points not collinear with $\textbf{x}$. Thus $(S_{U,W})_{\textbf{y}}=1$.

Since $\theta_{\alpha}$ induces $\begin{pmatrix}
                                        1&0\\\alpha&1
                                       \end{pmatrix}$ 
on $\textbf{x}^{\perp}/\textbf{x}$ it follows that $S_{U,W}$ fixes one line through $\textbf{x}$ and
permutes the remaining $q$ in orbits of size $p^k$.
\end{proof}

\begin{proof}[Proof of Theorem \ref{thm:secondthm}]
The group $E$ is the group defined in (\ref{eqE}), the group $P$ is as provided by Construction
\ref{con:P} and the group $S$ can be taken to be $S_{U,W}$ given by Construction \ref{con:SUW} with
$U$ a 1-dimensional subspace of $\mathrm{GF}(q)$ over $\mathrm{GF}(p)$. The theorem then follows
from Lemmas \ref{lem:E},\ref{lem:P}, \ref{lem:Pregular}, \ref{lem:SUW} and \ref{lem:SUWregular}, and
Corollaries \ref{cor:PnotE} and \ref{cor:SUPnotPE}.
\end{proof}

\section{Small generalised quadrangles}
\label{sec:smallGQ}

In general, the generalised quadrangle $\mathcal{Q}^x$ of order $(q-1,q+1)$ obtained by Payne
deriving $\mathsf{W}(3,q)$ has more regular subgroups than those exhibited in Section
\ref{sec:paynederived}. In this section we catalogue the point regular groups of automorphisms of
small generalised quadrangles. The results suggest that the problem is wild.

For $q=3$ we have $\mathcal{Q}^x\cong \mathsf{Q}^-(5,2)$ and so the point regular
groups of automorphisms are given by Theorem \ref{MainTheorem}. They are simply conjugates of the
groups $E$ and $P$ from Section \ref{sec:paynederived}.

For $q=4$ the full automorphism group of $\mathcal{Q}^x$ is $C_2^6\rtimes (3.A_6.2)$, which acts
transitively on the lines of $\mathcal{Q}^x$ (see \cite[\S V]{payne90}. Hence the dual of $\mathcal{Q}^x$ is a generalised
quadrangle of order $(5,3)$ with a point-transitive automorphism group.

\begin{example}
\label{eg:35}
Let $\mathcal{Q}^x$ be the generalised quadrangle of order $(3,5)$ obtained by Payne derivation from
$\mathsf{W}(3,4)$. A Magma \cite{magma} calculation\footnote{Our use of the computer was not
  complicated. We simply used the command \texttt{Subgroups(G:IsRegular)}, for the most part.}
reveals that $\Aut(\mathcal{Q}^x)$ has 58 conjugacy classes of regular subgroups with 30 different
isomorphism classes occuring.  In Table \ref{tab:qeq4} we document the number in the Small Group
Database of Magma of each isomorphism class and the number of conjugacy classes (indicated by the
symbol $\#$) of regular subgroups of that isomorphism type. We also give information about various
groups in the list and identify $E$, $P$ and the $S_{U,W}$. In this case all the $S_{U,W}$ are
conjugate in $\mathrm{Sp}(4,4)_{\textbf{x}}$. We note that $E$ is normal in $\Aut(\mathcal{Q}^x)$
and is the only abelian regular subgroup. The groups occuring have nilpotency class 1, 2, 3 or
4. Note that the list of regular groups includes a group isomorphic to a Sylow 2-subgroup of
$\mathrm{GL}(3,4)$ so it is possible for a Heisenberg group of even order to act regularly on the
points of a generalised quadrangle. This was previously believed to not be possible
\cite[p241]{conjecture}.
\end{example}

\begin{table}
 \caption{The isomorphism types of regular subgroups of the generalised quadrangle of order $(3,5)$}
\label{tab:qeq4}
\begin{tabular}{ccl||ccl||ccp{4cm}}
Group & \# & comment &Group & \# &comment &Group & \#& comment \\
\hline
 9& 1    &   &   90& 7     & &       215& 1& special \\
 18& 1   &   & 92& 1      &  &    219& 1 &special\\
 23& 6   &   &   102& 1    & &          224& 1 & special\\
 32& 5   &   &  136& 2     & &          226& 1 &special, $D_8\times D_8$\\
 33& 3  &    &  138& 2   &   &         227& 1 &special\\
 34& 1  &    &139& 2       & &          232& 2 &special\\
 35& 4   &   &193& 1       &$S_{U,W}$ &          241& 1 & special\\
 56& 1    & $P$  &   199& 1  &   &         242& 1 &special, Sylow 2-subgroup of $\mathrm{GL}(3,4)$\\
60& 4     &  &  202& 2  &   &         264& 1 &\\
 88& 1    &  &   206& 1   &  &        267& 1 & $E$
\end{tabular}
\end{table}

\begin{example}
Let $\mathcal{Q}$ be the generalised quadrangle of order $(5,3)$, the dual of the generalised
quadrangle of order $(3,5)$ in Example \ref{eg:35}.  Then $\mathcal{Q}$ has 96 points and 64 lines
and has automorphism group $C_2^6\rtimes (3.A_6.2)$ (see \cite{payne90}). It is known that $\Aut(\mathrm{Q})$ contains a
regular subgroup on points \cite[p46]{ow}.  In fact the automorphism group contains 6 different
conjugacy classes of regular subgroups on points (by a Magma calculation). They have shape as given
in Table \ref{tab:53}. By $2^{a+b}$ we mean a 2-group $P$ with center an elementary abelian group of
order $2^a$ and $P/Z(P)$ is elementary abelian of order $2^b$.

\begin{table}
 \caption{Regular subgroups of the generalised quadrangle of order $(5,3)$}
\label{tab:53}
\begin{center}
\begin{tabular}{l|ll}
Group & Shape & Notes\\
\hline
$H_1$ & $C_2^4\rtimes S_3$ & $Z(H_1)=1$, $H_1'=C_2^4\rtimes C_3$\\
$H_2$ & $2^{2+3}\rtimes C_3$& $Z(H_2)=1$, $H_2'=C_2^4$ \\
$H_3$ & $2^{2+3}\rtimes C_3$& $Z(H_3)=1$, $H_3'=C_4^2$ \\
$H_4$ & $C_4^2\rtimes S_3$ & $Z(H_4)=1$, $H_4'=C_4^2\rtimes C_3$\\
$H_5$ & $C_2^4\rtimes S_3$ & $|Z(H_5)|=2$, $H_5'=C_2^3\rtimes C_3$\\
$H_6$ & $2^{2+2}\rtimes S_3$& $|Z(H_6)|=2$, $H_6'=Q_8\rtimes C_3$\\
 \end{tabular}
\end{center}
\end{table}
\end{example}

For $q\geq 5$, the full automorphism group of $\mathcal{Q}^x$ is
$\mathrm{P}\Gamma\mathrm{Sp}(4,q)_{\textbf{x}}$ \cite{GJS}, which is not transitive on the lines of
$\mathcal{Q}^x$. In Table \ref{tab:smallq} we list, for certain values of $q$, the number of conjugacy classes of
point regular subgroups of $\Aut(\mathcal{Q})$ where $\mathcal{Q}$ is the generalised quadrangle of
order $(q-1,q+1)$ obtained from $\mathsf{W}(3,q)$ obtained by Payne derivation.

\begin{table}
\caption{Numbers of conjugacy classes of point regular subgroups of the Payne derived generalised quadrangle from $W(3,q)$}
\label{tab:smallq}
\begin{tabular}{ll|l}
$q$ & $\#$ regular subgroups & comments\\
\hline
2 & 4 & $2^3$, $C_4\times C_2$, 2 $D_8$'s,  conjugacy in $\mathrm{P}\Gamma\mathrm{Sp}(4,2)_x$\\
3 & 2 & this is $\mathsf{Q}^-(5,2)$ \\
4 & 58& 30 isomorphism classes\\
5 & 2 & $E$ and $P$.\\
7 & 2 & $E$ and $P$\\
8& 14 &  8 isomorphism types, nilpotency class 1 or 2\\
 &   & 2 conjugacy classes of subgroups isomorphic to $P$\\
 &    & 1 conjugacy class of $S_{U,W}$ with $U$ a 1-space\\
 &    & 2 conjugacy classes (but 1 isomorphism class) of $S_{U,W}$ with $U$ a 2-space\\
 &    & 2 further conjugacy classes of groups isomorphic to $S_{U,W}$ with $U$ a 2-space\\
9&  5& distinct isomorphism types, nilpotency class 2 or 3\\
 &   & 1 class of $S_{U,W}$\\
11& 2& $E$ and $P$\\
13& 2& $E$ and $P$\\
16& 231 &  1 conjugacy class of $S_{U,W}$ for $U$ a 1-space\\
  &     & 2 isomorphism (and conjugacy) classes of $S_{U,W}$ for $U$ a 3-space\\
  &     & 10 conjugacy classes of $S_{U,W}$ for $U$ a 2-space (all isomorphic)\\
  &     & nilpotency classes 1, 2, 3, 4, 5, 6 and 7.\\
17& 2& $E$ and $P$\\
19& 2& $E$ and $P$ \\
23& 2& $E$ and $P$\\
25 & 7 &  nilpotency class 2 or 3\\
     & &1 conjugacy class of $S_{U,W}$
\end{tabular}
\end{table}

\begin{example}
Let $\mathcal{Q}$ be the generalised quadrangle of order $(15,17)$ arising from the Lunelli-Sce
hyperoval. Then $\mathcal{Q}$ has 4096 points and 4608 lines. It follows from \cite{bms} that its
automorphism group is isomorphic to $G=2^{12}\rtimes H$ where $H$ is the stabiliser in
$\Gamma\mathrm{L}(3,16)$ of the hyperoval. The group $H$ has shape $(3_+^{1+2}\times C_5)\rtimes
(C_8\times C_2)$.

A Magma \cite{magma} calculation shows that the group $G$ contains 54 conjugacy classes of groups
regular on points. This includes the elementary abelian 2-group which is the socle of $G$. There are
16 isomorphism types of groups and one further conjugacy class of subgroups for which we are unable
to determine whether they are isomorphic to any of the former $16$ types. There are no special
2-groups on the list. The nilpotency classes of the groups are 1, 2, 3, 4 and 7. There are two
conjugacy classes of groups of nilpotency class 7 and all such groups are isomorphic. Every member
of this isomorphism class has centre of order 2, derived subgroup of order $2^7$ and exponent 16.

It is proved in \cite[p46]{ow} that the generalised quadrangle of order $(17,15)$ arising as the
dual of $\mathcal{Q}$ has no point regular groups of automorphisms. This was confirmed by computer
calculations.
\end{example}

\section{Conjectures in the literature}
\label{sec:conjectures}

We note that the second and third examples of Theorem \ref{MainTheorem} appear to have been
overlooked in the classsification of subgroups of $\mathrm{P\Gamma U}(4,q)$ transitive on lines
given in \cite[Corollary 5.12]{KantLieb}.

In \cite{conjecture} the authors prove 

\begin{theorem}
\label{thm:TdW}
Let $\mathcal{Q}$ be a generalised quadrangle of order $(s, t)$ admitting a point regular group $G$,
where $G$ is a $p$-group and $p$ is odd. Suppose $|Z(G)|\geq \sqrt[3]{|G|}$. Then the following
properties hold.
\begin{enumerate}
 \item We have $t = s + 2$, and there is a generalised quadrangle $\mathcal{Q}'$ of order $s + 1$ with a regular
point $x$, such that $\mathcal{Q}$ is Payne derived from $\mathcal{Q'}$ with respect to $x$. The
generalised quadrangle $\mathcal{Q}'$ is an elation generalised quadrangle with  elation group $K$ isomorphic to $G$.
\item We have $|Z(G)| = \sqrt[3]{|G|}$, that is, $|Z(G)| = s + 1.$
\end{enumerate}
\end{theorem}

For $q$ odd, the groups $P$ and $S$ from Theorem \ref{thm:secondthm} satisfy the hypotheses of
Theorem \ref{thm:TdW} with $\mathcal{Q}$ being a generalised quadrangle of order $(q-1,q+1)$. Hence
there should be generalised quadrangles of order $q$ with elation groups $K$ and $K'$ isomorphic to
$P$ and $S$ respectively. However, the only generalised quadrangle of order $3$ is $\mathsf{W}(3,3)$
\cite[\S 6.2]{paynethas} and this does not have an elation group isomorphic to $P$.  The theorem
seems to also fail with respect to $P$ for larger values of $q$.

Thus we also have a counterexample to the following conjecture of \cite{conjecture}.
\begin{quote}\textit{
\textbf{Conjecture 1:} If $\mathcal{Q}$ is a generalised quadrangle admitting a point regular group
of automorphisms $G$, then there exists an elation generalised quadrangle $\mathcal{Q}'$ of order
$s$ with elation group $G'$, such that $\mathcal{Q}$ can be obtained from $\mathcal{Q}'$ by Payne
derivation with respect to $x$, and such that $G\cong G'$ .  }
\end{quote}

In \cite{conjecture} the authors also make the following conjecture

\begin{quote}\textit{
\textbf{Conjecture 2:} If a finite generalised quadrangle admits a point regular group of
automorphisms $G$, such that $G$ is a $p$-group, $p$ odd, with the property that
$|\mathsf{Z}(G)|\geqslant \sqrt[3]{|G|}$, then $G$ is isomorphic to a Heisenberg group of dimension
3 over $\mathrm{GF}(q )$, where $q$ is a power of $p$.}
\end{quote}
Moreover, in \cite[Conjecture 4.4.1]{ow} the following more general conjecture is made:

\begin{quote}\textit{
\textbf{Conjecture 3:} If a finite thick generalised quadrangle $\mathcal{Q}$ admits a group of
automorphisms $G$ which acts regularly on the set of points, then either $\mathcal{Q}$ is the
generalised quadrangle of order $(5, 3)$, or $G$ is $(1)$ an elementary abelian 2-group, or $(2)$ an
odd order Heisenberg group, and in (1)-(2) $\mathcal{Q}$ is a Payne derived generalised quadrangle
arising in the usual way from an elation generalised quadrangle with elation group isomorphic to
$G$.}
\end{quote}
The authors state in \cite{ow} that perhaps `Heisenberg' could be replaced by `special' in the above
conjecture.  The groups $P$ for $q=3^f$ and $S_{U,W}$ for $q$ odd and not a prime are not Heisenberg
groups and so are counterexamples to Conjectures 2 and 3. The groups $S_{U,W}$ are not
special. Moreover, when $q$ is even, $P$ and $S_{U,W}$ are nonabelian 2-groups acting regularly on a
generalised quadrangle and so are further counterexamples.

The example of $C_{513}\rtimes C_9$ acting regularly on the points of $Q^-(5,8)$ is a particularly
interesting counterexample to Conjecture 3 as $Q^-(5,8)$ is not Payne derived and the group is not
nilpotent. It is however meta-abelian, but we saw in Example \ref{eg:35} that there are 4 
groups acting regularly on the generalised quadrangle of order $(3,5)$ that are not meta-abelian.

\section*{Acknowledgements}
The authors would like to thank Frank De Clerck, Bill Kantor, Tim Penttila, Gordon Royle and Pablo
Spiga for their comments on a preliminary draft, and the anonymous referees.  This work was
supported by an Australian Research Council Discovery Projects Grant (no.~DP0984540) and the second
author is also supported by an Australian Research Fellowship.

\end{document}